\documentclass[acmlarge]{acmart}
\AtBeginDocument{%
  \providecommand\BibTeX{{%
    \normalfont B\kern-0.5em{\scshape i\kern-0.25em b}\kern-0.8em\TeX}}}

\usepackage{amsmath}
\usepackage{amsthm}
\usepackage{dsfont}
\newtheorem{remark}{remark}
\newtheorem{prob}{Problem}
\newtheorem{assumption}{Assumption}
\usepackage{algorithm}
\usepackage{algpseudocode}
\usepackage{textcomp}
\usepackage{multirow}

\begin{document}

\title{ An Efficient Vertex Retrieval Algorithm for Ball Polyhedra }
%
\author{Marius Costandin}
\email{costandinmarius@gmail.com}
\affiliation{\country{Romania}}

\begin{abstract}
We consider the intersection $\mathcal{Q}$ of $m$ equal-radius balls in $\mathbb{R}^n$ whose centers lie on a common sphere centered at the origin, with the origin contained in the convex hull of the centers. Assuming $\mathcal{Q}$ possesses a unique vertex $x^{\star}$ lying on the hyperplane $1_{n \times 1}^T \cdot x=0$, under a combinatorial-stability assumption, we give a polynomial-time algorithm that recovers $x^\star$. The algorithm constructs a one-parameter family of auxiliary polytopes (max-indicator polytopes) whose vertices move radially under a controlled translation of a quadratic center. The distinguished vertex remains fixed, while all others travel along fixed rays. By comparing a pair of symmetrically perturbed polytopes with the unperturbed one and solving a linear number of linear programs, the algorithm isolates $x^\star$.   
\end{abstract}

%



\maketitle

\section{Introduction}

The study of ball-polyhedra, defined as the non-empty intersection of a finite number of closed balls of equal radii in Euclidean space, represents a fundamental area within discrete geometry and optimization. 

Intersections of finitely many congruent Euclidean balls, commonly referred to as \emph{ball-polyhedra}, constitute a natural nonlinear analogue of convex polyhedra. Their geometric and combinatorial properties have received considerable attention; foundational aspects of their structure were studied in \cite{BezdekLangiNaszodiPapez2007}, \cite{BezdekNaszodi2006}, \cite{BezdekNaszodi2013}, while more recent work has further developed the theory of intersections of balls and related classes of ball-polyhedra \cite{BezdekLangiNaszodi2025}, \cite{Bezdek2026}. Among the questions arising in this setting are problems concerning the structure and identification of faces and vertices.

Ball-polyhedra form a natural nonlinear analogue of convex polyhedra and arise in a variety of geometric and optimization settings. Importantly, they include realizations of the $\{\pm 1\}^n$ cube as special cases: when the centers and common radius are chosen appropriately, the resulting intersection of congruent balls can recover the combinatorial type and the vertex set of a cube.

For a general polytope presented by inequalities, deciding whether a given hyperplane contains some vertex, or optimizing a linear functional over the vertex set subject to a linear equality constraint, is NP-hard in the worst case.

In this paper, we consider the following problem. 

\begin{prob}\label{P1}
For $m,n \in \mathbb{N}$, given an intersection of $m > n$ balls $\mathcal{Q} \subseteq \mathbb{R}^n$, $\left \lfloor \frac{n}{2}\right\rfloor + 1$ hyperplanes $\mathcal{H}_i$ and $C_S \in \mathbb{R}^n$ with:
\begin{enumerate}
\item the intersecting balls have equal radii and the centers $C_k$ lie on a sphere centered in origin
\item the space origin belongs to their convex hull
\item the hyperplanes pass through the origin
\item the hyperplanes pass through a distinguished vertex, $x^\star$ of the intersection of balls
\item the hyperplanes pass through no other vertex, except possibly $-x^{\star}$ if it is a vertex 
\item any set of $n$ vectors from $\{C_S - C_{k}| k \in \overline{1,m}\} \cup \{h_i | i \in \overline\{ 1,\left \lfloor \frac{n}{2}\right\rfloor + 1\}\}$ is linear independent
\item the distance $\|C_S - x^{\star}\| = r_S$ is given. 
\end{enumerate} find $x^{\star}$! 

\end{prob}

Note that Problem \ref{P1} is not trivial, because knowing roughly $\frac{n}{2} + 1$ hyperplanes and a sphere containing a point does not lead trivially to the retrieval of the point.  

\section{Intersection of one Hyperplane with $\mathcal{Q}$}
 First we focus on one hyperplane and assume w.l.o.g that it is $\mathcal{H} = \{x | 1_{n \times 1}^T \cdot x = 0\}$. Note that any other general case can be reduced to this one by suitable space rotations. 

 Since $C_S \in \text{conv}\{C_1, \hdots, C_m\}$ define the following functions
\begin{align}
h(x) = \max_{k \in \{1, \hdots, m\} } \|x - C_{k}\|^2 - R^2 \hspace{0.5cm} g(x) = \|x - C_S\|^2.
\end{align} Note that
\begin{align}
&h(x) - g(x) = \nonumber \\
& \max_{k \in \{1, \hdots, m\} } 2 \cdot (C_S - C_{k})^T \cdot x + \|C_{k}\|^2 - R^2 - \|C_S\|^2
\end{align} i.e. the quadratic terms cancel out, leaving $h-g$ linear in $x$. 

 Because $C_S$ belongs to the (interior of the) convex hull of the points $C_1, \hdots, C_m$, one obtains that the polyhedral set family below contains only polytopes, see \cite{Costandin2024}.
\begin{align}
\mathcal{P}_{C_S,r^2} = \{x | h(x) - g(x) \leq -r^2 \}
\end{align} where, for a fixed balls centers and radii, we indicate that each member of the family depends on the point $C_S$ and  a real parameter $r$. 

The main result in \cite{Costandin2024} is a characterization of the farthest in $\mathcal{Q}$ (the intersection of balls) to the point $C_S$ in terms of the evolution of the polyhedral family $\mathcal{P}_{C_S,r^2}$ with varying $r$. 


\subsection{Analysis of max-indicators polytopes} \label{ss_vertex_locus}
We show that for $r = 0$ the polytope $\mathcal{P}_{C_S, 0}$ includes the intersection of balls $\mathcal{Q}$. Indeed, let $x \in \partial \mathcal{Q}$, then $h(x) = 0$ and obviously $-g(x) \leq 0$ hence $x \in \mathcal{P}_{C_S, 0}$. 

Increasing $r$, the polytopes shrink, and for each polytope vertex $v_i$ exists $r_i$, such that for that specific $r_i$, on its way down, the polytope vertex passes through the vertex of the intersection of balls. Indeed, let $w_i$ be a vertex of the ball-polyhedra $\mathcal{Q}$ and let $R_i = g(w_i)$.  Since $w_i$ is a vertex of $\mathcal{Q}$, it follows that $h(w_i) = 0$, therefore letting $r = R_i$ follows that $h(w_i) - g(w_i) = h(w_i) - R_i^2 = -R_i^2 = -r^2$ hence $w_i \in \partial \mathcal{P}_{C_S, r^2}$. This shows that the vertices of $\mathcal{Q}$ can be characterized as intersections of facets of $\mathcal{P}_{C_S, r^2}$. 

It has been shown in \cite{Costandin2024} that the value of $r$ corresponding to the last vertex to enter $\mathcal{Q}$, is the maximum distance to $C_S$ in $\mathcal{Q}$ and the corresponding vertices are the extreme points. However, we do not pursue that path here. 

Therefore, increasing $r$, but keeping $C_S$ fixed, the vertices of $\mathcal{P}_{C_S,r^2}$ slide on some lines. We show below, that if one modifies $C_S \to D_S$ and keeps $r$ constant, the vertices of the resulting polytope, $\mathcal{P}_{D_S, r^2}$ will be different, but they are on the same lines the vertices of $\mathcal{P}_{C_S, r^2}$ slide. This means that the vertices of the max-indicator polytopes, as we call these polytopes, slide on the same lines ,that we shall refer to as the "vertex locus". 

%
%
In the following, we use these polytopes to search for the special vertex of $\mathcal{Q}$ on the said intersection of hyperplanes.
 
We begin with a general result about the polytopes $\mathcal{P}_{C_S,r^2}$.
 
\begin{lemma}\label{L3.5b}
For $r>0$ let $x$ be a vertex of $\mathcal{P}_{C_S,r^2}$. Assume w.l.o.g. that $I = \{k_1, \hdots, k_n\}$ is the active facet set defining the nondegenerate vertex $x$. 

Consider $D_S \in \mathbb{R}^n$, a perturbation of $C_S$, and $r_D \in \mathbb{R}$ such that the polytopes $\mathcal{P}_{C_S, r^2}$ and $\mathcal{P}_{D_S, r_D^2}$ are combinatorially equivalent i.e. the corresponding facets indexed by $I$ remain active and linearly independent. Then for
\begin{align}\label{E2.4d}
r_D^2 = r^2 + \|D_S\|^2 - \|C_S\|^2 - 2 \cdot (D_S - C_S)^T\cdot x > 0
\end{align}  $x$ is also a vertex of $\mathcal{P}_{D_S,r_{D}^2}$. 
\end{lemma}

\begin{proof}
Let $x$ be a vertex of $\mathcal{P}_{C_S,r^2}$ formed by the balls centers $C_{k_1}, \hdots, C_{k_n}$
\begin{align}\label{E2.5c}
&\begin{bmatrix} 2 \cdot (C_S - C_{k_1})^T\\ \vdots \\ 2 \cdot (C_S - C_{k_n})^T\end{bmatrix} \cdot x = \nonumber \\
& = \begin{bmatrix} -r^2 + \|C_S\|^2\\ \vdots \\ -r^2 + \|C_S\|^2  \end{bmatrix} + \begin{bmatrix} R^2 - \|C_{k_1}\|^2\\ \vdots \\ R^2 - \|C_{k_n}\|^2 \end{bmatrix} = \nonumber \\
& = (-r^2 + \|C_S\|^2 + R^2 - \rho^2) \cdot 1_{n \times 1}
\end{align} since we assumed that $\|C_{k_1}\| = \hdots = \|C_{k_m}\| = \rho$. Let $y$ be the corresponding (generated by the same balls centers) vertex in the polytope $\mathcal{P}_{D_S, r_D^2}$:

\begin{align}\label{E2.6c}
&\begin{bmatrix} 2 \cdot (D_S - C_{k_1})^T\\ \vdots \\ 2 \cdot (D_S - C_{k_n})^T\end{bmatrix} \cdot y = \nonumber \\
& =  \begin{bmatrix} -r_D^2 + \|D_S\|^2\\ \vdots \\ -r_D^2 + \|D_S\|^2  \end{bmatrix} + \begin{bmatrix} R^2 - \|C_{k_1}\|^2\\ \vdots \\ R^2 - \|C_{k_n}\|^2 \end{bmatrix} \nonumber \\
& \begin{bmatrix} 2 \cdot (C_S - C_{k_1})^T\\ \vdots \\ 2 \cdot (C_S - C_{k_n})^T\end{bmatrix} \cdot y = \begin{bmatrix} -r_D^2 + \|D_S\|^2\\ \vdots \\ -r_D^2 + \|D_S\|^2  \end{bmatrix}  \nonumber \\
& + \begin{bmatrix} R^2 - \|C_{k_1}\|^2\\ \vdots \\ R^2 - \|C_{k_n}\|^2 \end{bmatrix} - 2\cdot \begin{bmatrix} (D_S - C_S)^T\\ \vdots \\ (D_S - C_S)^T \end{bmatrix} \cdot y \nonumber \\
\end{align}  

Choose $r_D$ in (\ref{E2.6c}) such that 
\begin{align}\label{E3.20a}
-r_D^2 &+ \|D_S\|^2 - 2\cdot (D_S - C_S)^T\cdot x = -r^2 + \|C_S\|^2 \iff \nonumber \\
r_D^2 &= r^2 + \|D_S\|^2 - \|C_S\|^2 - 2 \cdot (D_S - C_S)^T\cdot x 
\end{align} then $x$ is also a solution to (\ref{E2.6c}).  
\end{proof}

We give the following remarks: 
\begin{remark}
One can see $\mathcal{P}_{D_S,r_D^2}$ as a perturbation of $\mathcal{P}_{C_S,r^2}$. These two polytopes indeed have different normal facet vectors i.e. $D_S - C_{k_i}$ for $\mathcal{P}_{D_S,r_D^2}$ and $C_S - C_{k_i}$ for $\mathcal{P}_{C_S,r^2}$, yet their corners, as $r$ respectively $r_D$ vary, translating the facets, \textbf{slide on the same lines}.
\end{remark}
\begin{remark}
 Note that although all the corners of the two polytopes slide on the same lines, they do so independently, i.e. the parameter $r_D$ depends on the chosen corner. These lines, therefore, can be seen as some \textbf{vertex locus} and depend on the set $Q$ only, not on the point $C_S$, to which we are trying to find the farthest in $\mathcal{Q}$. 
\end{remark} 

\subsection{Usage of max-indicators polytopes for vertex search}

For $C_S$ and $r_S$ given, the point $x^{\star}$ is a vertex of the polytope $\mathcal{P}_{C_S,r_S^2}$. Define
\begin{align}\label{E2.8}
C_S(\zeta) &= C_S + \zeta \cdot 1_{n\times 1} \nonumber \\
r(\zeta)^2 &= r_S^2 + \|C_S(\zeta)\|^2 - \|C_S(0)\|^2.
\end{align} 

The rest of the paper continues under the assumptions: 
\begin{assumption} \label{A1}
 For $C_S$ and $r_S$ exists and is known $ \epsilon > 0$, such that the polytopes $\mathcal{P}_{C_S(\zeta), r^2(\zeta)}$ are combinatorially equivalent to $\mathcal{P}_{C_S,r_S}$ for all $\zeta \in \left[ \frac{-\epsilon}{\sqrt{n}}, \frac{\epsilon}{\sqrt{n}}\right]$. 
\end{assumption}

\begin{assumption} \label{A2}
For all $k_i \in \overline{1, ..., m}$ one has 
\begin{align}
&\{C_S - C_{k_i}| i \in \overline{1, n}\} \hspace{0.5cm} \text{is linear independent} \Rightarrow \nonumber \\
& \{C_S - C_{k_i}| i \in \overline{1, n}, i\neq j\} \cup \{ 1_{n \times 1}\}  \hspace{0.3cm}  \text{is lin. indep.}
\end{align} for all $j \in \overline{1,n}$
\end{assumption}

That is, we require that some sufficiently small perturbations of the resulting polytopes are combinatorially equivalent. 

In the following lemma, we establish how the vertices of the polytopes $\mathcal{P}_{C_S(\zeta), r(\zeta)^2}$ move with varying $\zeta$. 

\begin{lemma}\label{L2.4}
For a fixed $r(0) = r$, let $y$ be a vertex of $\mathcal{P}_{C_S(0), r(0)^2}$ and $z$ the combinatorically equivalent perturbed vertex in the polytope $\mathcal{P}_{C_S(\zeta), r(\zeta)^2}$, i.e $y = z(0)$. Then 

\begin{align}
z(\zeta) = \left( 1 -  \frac{2 \cdot \zeta \cdot \lambda}{1 + 2\cdot \zeta \cdot \lambda} \right) \cdot y
\end{align}
where 
\begin{align}
\Lambda \cdot y &= (-r^2+\|C_S\|^2 + R^2-\rho^2) \cdot 1_{n\times 1} \nonumber \\
 \lambda &= 1_{n\times 1}^T \cdot \Lambda^{-1} \cdot 1_{n\times 1} 
\end{align} and $\Lambda \in \mathbb{R}^{n\times n}$ is the system matrix, see (\ref{E2.5c}). 
\end{lemma}
\begin{proof} 

Let us now analyze the vertex $\Lambda \cdot y = \left( -r^2+\|C_S\|^2+ R^2 - \rho^2 \right) \cdot 1_{n\times 1}$ in $\mathcal{P}_{C_S, r^2}$. The corresponding vertex in $\mathcal{P}_{C_S(\zeta), r(\zeta)^2}$ is found similar to (\ref{E2.6c}) as a rank one update to the vertex system matrix

\begin{align}\label{E2.26}
&\left(\Lambda + 2\cdot 1_{n\times 1} \cdot (-C_S + C_S(\zeta))^T \right) \cdot z = \nonumber \\
&= \left( -r(\zeta)^2+\|C_S(\zeta)\|^2 + R^2 - \rho^2 \right) \cdot 1_{n\times 1}  \nonumber \\
& \left(\Lambda + 2\cdot \zeta \cdot 1_{n\times 1} \cdot 1_{n\times 1}^T \right) \cdot z = \nonumber \\
&= \left( -r(\zeta)^2+\|C_S(\zeta)\|^2 + R^2 - \rho^2  \right) \cdot 1_{n\times 1} \nonumber \\
& = \Lambda \cdot y
\end{align}

As such from (\ref{E2.26}) we get
\begin{align}\label{E2.28}
z = \left( \Lambda + 2\cdot \zeta \cdot 1_{n \times 1}\cdot 1_{n \times 1}^T \right)^{-1} \cdot \Lambda \cdot y 
\end{align}

Using Sherman-Morrison inversion formula, it is obtained
\begin{align}\label{E2.13c}
 &\left( \Lambda + 2\cdot \zeta \cdot 1_{n \times 1}\cdot 1_{n \times 1}^T \right)^{-1}  = \nonumber \\
& =  \Lambda^{-1} - 2\cdot \zeta \cdot \frac{\Lambda^{-1} \cdot 1_{n\times 1} \cdot 1_{n \times 1}^T \cdot \Lambda^{-1}}{1 + 2\cdot \zeta \cdot 1_{n\times 1}^T \cdot \Lambda^{-1} \cdot 1_{n \times 1}}
\end{align} Let 
\begin{align}\label{E2.14c}
M &= \Lambda^{-1} \cdot 1_{n\times 1} \cdot 1_{n \times 1}^T \hspace{0.5cm} \lambda = 1^T_{n\times 1} \cdot \Lambda^{-1} \cdot 1_{n\times 1} \hspace{0.5cm} 
\end{align} then we obtain from (\ref{E2.13c}), (\ref{E2.14c}) and (\ref{E2.28})

\begin{align}\label{E2.31}
z &= y - 2 \cdot \zeta \cdot \frac{M \cdot y}{1 + 2\cdot \lambda \cdot \zeta} 
\end{align} 

However since $y = (-r^2 + \|C_S\|^2 + R^2 - \rho^2) \cdot \Lambda^{-1}\cdot 1_{n \times 1}$ 
Let us define
\begin{align}
\Psi = -r^2 + \|C_S\|^2 + R^2 - \|C_k\|^2
\end{align}  hence $y = \Psi \cdot \Lambda^{-1}\cdot 1_{n \times 1}$. One has
\begin{align} \label{E2.32}
M \cdot y &= \Lambda^{-1}\cdot 1_{n\times 1} \cdot 1_{n \times 1}^T \cdot y  \nonumber \\
& =  \frac{1}{\Psi} \cdot y \cdot 1_{n\times 1}^T \cdot \Psi \cdot \Lambda^{-1} \cdot 1_{n\times 1} = \lambda \cdot y
\end{align} 
then from (\ref{E2.14c}) and (\ref{E2.32}) it is obtained

\begin{align}
\lambda = \frac{1_{n\times 1}^T\cdot y}{\Psi}\hspace{0.5cm} 
M\cdot y = \frac{1_{n\times 1}^T \cdot y }{\Psi} \cdot y = \lambda \cdot y
\end{align}

therefore (\ref{E2.31}) becomes 

\begin{align}
z(\zeta) &= \left( 1 -  \frac{2 \cdot \zeta \cdot \frac{1_{n\times 1}^T \cdot y}{\Psi}}{1 + 2\cdot \zeta \cdot \frac{  1_{n\times 1}^T \cdot y}{\Psi}} \right) \cdot y \nonumber \\
&= \left( 1 -  \frac{2 \cdot \zeta \cdot \lambda}{1 + 2\cdot \zeta \cdot \lambda} \right) \cdot y = \frac{1}{1 + 2 \cdot \zeta \cdot \lambda} \cdot y
\end{align} which concludes the proof. 
\end{proof}

We'll give now a series of remarks on the interpretation of the above lemma. 
\begin{remark}\label{E2.4}
In the above Lemma \ref{L2.4} one can see that the vertex locus for each vertex, is a radial line. This together with Subsection \ref{ss_vertex_locus} characterizes the vertex locus as starting from origin and passing through the vertex of $\mathcal{Q}$. As such, if there is a vertex $x$ of $\mathcal{Q}$ meeting $x^T\cdot 1_{n \times 1} = 0$ then the vertex locus itself belongs to the said hyperplane. 

The contrary is true as well, if the vertex of $\mathcal{Q}$ does not belong to the hyperplane, neither does the vertex locus for that vertex. 
\end{remark}

\begin{remark}
From the above Remark \ref{E2.4} the polytope $\mathcal{P}_{C_S, r^2}$, has exactly one vertex $x^{\star}$ with $1_{n\times 1}^T\cdot x = 0$, i.e. the one on the single vertex locus in the said hyperplane. 

This together with the definition of $r(\zeta)$ in (\ref{E2.8}) which matches (\ref{E2.4d}), assures that $x^{\star}$ is a vertex in all $\mathcal{P}_{C_S(\zeta), r(\zeta)^2}$. Indeed, since $ 1_{n\times 1}^T \cdot x^{\star} = 0$ follows $\lambda = 0$ hence from Lemma \ref{L2.4} follows that $x^{\star}(\zeta) = x^{\star} = x^{\star}(0)$. 
\end{remark}

\begin{remark} We have shown that the vertices $x$ of $\mathcal{P}_{C_S,r^2}$ which do not meet $1_{n\times 1}^T \cdot x = 0$ will be translated on their corresponding vertex locus as $C_S$ is moved along the direction $1_{n \times 1}$, while the unique vertex $x^{\star}$ which meets $1_{n \times 1}^T \cdot x^{\star} = 0$ will remain fix under such perturbation. 
\end{remark}


The following lemma gives a technical characterization of the dynamics of the moving vertices under varying $\zeta$ that shall be used in the proof of the main theorem. In particular, the unperturbed vertex is written as a convex combination of the symmetrically perturbed corresponding vertices. The following lemma gives a lower bound on a function of the coefficients of this convex combination. 

\begin{lemma}\label{L2.9}
Let $q = \frac{\|z(\zeta)\|}{\|z(-\zeta)\|}$ and $\psi \in [0,1]$ such that $y = \psi\cdot z(-\zeta) + (1-\psi)\cdot z(\zeta)$. Then  if $1_{n\times 1}^T\cdot y \neq 0$ one has
\begin{align} \label{E2.20c}
\psi + (1-\psi)\cdot q + \frac{\psi}{q} + (1-\psi) = 2.
\end{align} 
\end{lemma}
\begin{proof}
\begin{align}\label{E2.36}
\frac{\|z(\zeta)\|}{\|z(-\zeta)\|} &= \frac{ 1 -  \frac{2 \cdot \zeta \cdot \lambda}{1 + 2\cdot \zeta \cdot \lambda} }{1 +  \frac{2 \cdot \zeta \cdot \lambda}{1 - 2\cdot \zeta \cdot \lambda}  } =\frac{1 - 2\cdot \zeta \cdot \lambda}{1 + 2\cdot \zeta \cdot \lambda} =: q
\end{align} 

One can find $\psi \in [0,1]$ with $y = \psi\cdot z(-\zeta) + (1-\psi)\cdot z(\zeta)$ as follows
\begin{align}
 y =& \psi \cdot \left( 1 +  \frac{2 \cdot \zeta \cdot \lambda}{1 - 2\cdot \zeta \cdot \lambda} \right) \cdot y + \nonumber \\
& + (1-\psi)\cdot \left( 1 -  \frac{2 \cdot \zeta \cdot \lambda}{1 + 2\cdot \zeta \cdot \lambda} \right) \cdot y
\end{align} therefore letting $2\cdot \zeta \cdot \lambda = x$ one obtains

\begin{align}
y = \psi \cdot \frac{1}{1 - x} \cdot y + (1- \psi) \cdot \frac{1}{1 + x} \cdot y
\end{align} then 
\begin{align}
1 - \frac{1}{1+x} = \psi \cdot \left(\frac{1}{1-x} - \frac{1}{1+x} \right)
\end{align} hence 
\begin{align}\label{E2.26w}
\psi = \frac{\frac{x}{1+x}}{\frac{1}{1-x} - \frac{1}{1+x} } = \frac{\frac{x}{1+x}}{ \frac{2x}{1-x^2} } = \frac{1-x}{2} = \frac{1 - 2\cdot \zeta \cdot \lambda}{2}
\end{align}

and verify that
\begin{align} \label{E2.41b}
&\psi + (1-\psi)\cdot q + \frac{\psi}{q} + (1-\psi) =  q - \psi\cdot q + \frac{\psi}{q} + 1 = 2
\end{align} where $q = \frac{1-x}{1+x}$ and $\psi = \frac{1-x}{2}$ with $x = 2\cdot \zeta \cdot \lambda \in \mathbb{R}$.
\end{proof}

Finally, we prove that the intersection of the polytopes $\mathcal{P}_{C_S(\zeta), r(\zeta)^2}$ and  $\mathcal{P}_{C_S(-\zeta), r(-\zeta)^2}$ is  included in the polytope  $\mathcal{P}_{C_S(0), r(0)^2}$ with the exception of the points which belong simultaneously to both boundaries of $\mathcal{P}_{C_S(\zeta), r(\zeta)^2}$ and  $\mathcal{P}_{C_S(-\zeta), r(-\zeta)^2}$. Those shall be on the boundary of $\mathcal{P}_{C_S(0), r(0)^2}$ as well.

\begin{theorem}\label{T2.9}
The following are true 
\begin{align}
\mathcal{P}_{C_S(-\zeta), r(-\zeta)^2} \cap \mathcal{P}_{C_S(\zeta), r(\zeta)^2} \subseteq  \mathcal{P}_{C_S(0), r(0)^2}
\end{align} 

Moreover,
\begin{align}
& x \in \partial \mathcal{P}_{C_S(-\zeta), r(-\zeta)^2} \cap \partial \mathcal{P}_{C_S(-\zeta), r(-\zeta)^2}  \iff  \nonumber \\
& x \in  \partial \mathcal{P}_{C_S(-\zeta), r(-\zeta)^2} \cap \partial \mathcal{P}_{C_S(-\zeta), r(-\zeta)^2}   \cap \partial \mathcal{P}_{C_S(0), r(0)^2} 
\end{align} 
\end{theorem}
\begin{proof}
Assume w.l.o.g. that $\underline{y}_1, \hdots, \underline{y}_n$ are some vertices in $\mathcal{P}_{C_S(-\zeta), r^2(-\zeta)}$ such that 
\begin{align}
x = \sum_{k=1}^n \alpha_k \cdot \underline{y}_k \hspace{0.5cm} \text{with  }\alpha_k \geq 0 \hspace{0.5cm} \sum_{k=1}^n \alpha_k \leq 1
\end{align} then denote the corresponding vertices in $\mathcal{P}_{C_S(\zeta), r^2(\zeta)}$ as $\overline{y}_1, \hdots, \overline{y}_n$ and one has
\begin{align}
x = \sum_{k=1}^n \beta_k \cdot \overline{y}_k \hspace{0.5cm} \text{with   }\beta_k \geq 0 \hspace{0.5cm} \sum_{k=1}^n \beta_k \leq 1.
\end{align} Letting $y_1, \hdots, y_n$ denote the vertices in $\mathcal{P}_{C_S(0), r(0)^2}$ we write 
\begin{align}\label{E2.47}
x = \sum_{k=1}^n \gamma_k \cdot y_k \hspace{0.5cm} \text{and ask  } \gamma_k \geq 0 \hspace{0.5cm} \sum_{k=1}^n \gamma_k <^{?} 1.
\end{align} Indeed, the above are true because $0_{n \times 1}$ being in the polytopes forms a pyramid together with some other $n$ vertices. For each point exists such a pyramid which contains it. The point is then written as the convex combination of the pyramid vertices. When the vertices are scaled up or down moving from one polytope to another, the point remains in the cone detemrined by this vertices.

Note that $y_k = \psi_k \cdot \underline{y}_k + (1- \psi_k)\cdot \overline{y}_k$ where from (\ref{E2.26w}) one has $\psi_k = \frac{1 - 2 \cdot \zeta \cdot \lambda_k}{2}$ and $\frac{\|\overline{y}_k\|}{\|\underline{y}_k\|} = q_k = \frac{1 - 2 \cdot \zeta \cdot \lambda_k}{1 + 2 \cdot \zeta \cdot \lambda_k}$, see (\ref{E2.36}). 

As such (\ref{E2.47}) becomes 
\begin{align}
x &= \sum_{k=1}^n \gamma_k \cdot \left( \psi_k \cdot \underline{y}_k + (1-\psi_k)\cdot \overline{y}_k\right) \nonumber \\
& = \sum_{k=1}^n \gamma_k \cdot \left( \psi_k  + (1-\psi_k)\cdot q_k\right) \cdot \underline{y}_k \hspace{0.5cm} \nonumber \\
& \hspace{0.5cm} \text{i.e} \hspace{0.5cm} \alpha_k =  \gamma_k \cdot \left( \psi_k  + (1-\psi_k)\cdot q_k\right) \nonumber \\
& = \sum_{k=1}^n \gamma_k \cdot \left( \psi_k \cdot \frac{1}{q_k} + (1-\psi_k) \right) \cdot \overline{y}_k  \hspace{0.3cm} \nonumber \\
& \hspace{0.5cm} \text{i.e} \hspace{0.5cm} \beta_k =  \gamma_k \cdot \left( \psi_k \cdot \frac{1}{q_k} + (1-\psi_k)\right) 
\end{align} Let $a_k =  \psi_k  + (1-\psi_k)\cdot q_k$ and $b_k = \psi_k \cdot \frac{1}{q_k} + (1-\psi_k)$ and we have
\begin{align}
1 &\geq \sum_{k=1}^n \alpha_k = \sum_{k=1}^n \gamma_k \cdot a_k \nonumber \\
1 &\geq \sum_{k=1}^n \beta_k = \sum_{k=1}^n \gamma_k \cdot b_k
\end{align} therefore
\begin{align}
\sum_{k=1}^n  \gamma_k \cdot \frac{a_k + b_k}{2} \leq 1
\end{align}

Finally, since $\gamma_k \geq 0$ and from (\ref{E2.20c}) follows that $\frac{a_k + b_k}{2} = 1$ we get $ \sum_{k=1}^n \gamma_k < 1$ i.e. the point is in the relative interior of the unperturbed polytope if $\sum_l \alpha_k < 1$ or $\sum_l \beta_k < 1$, that is if the point is in the relative interior of any perturbed polytope. Furthermore, we get $ \sum_{k=1}^n \gamma_k = 1$, i.e. the point is on the boundary of the unperturbed polytope if $\sum_l \alpha_k = 1$ and $\sum_l \beta_k = 1$ i.e. the point is on the boundary of both perturbed polytopes. 
\end{proof}

\begin{remark} The points that belong to both boundaries are the points on the hyperplane $1_{n \times 1}^T \cdot x = 0$. Indeed, assume $x$ meets
\begin{align}
&2 \cdot (C_S - C_{k_j})^T \cdot x = 0 \nonumber \\
&2 \cdot (C_S + \zeta \cdot \frac{1_{n\times 1}}{\sqrt{n}} - C_{k_j})^T \cdot  x = 0 \nonumber \\
&2 \cdot (C_S - \zeta \cdot \frac{1_{n\times 1}}{\sqrt{n}} - C_{k_j})^T \cdot x = 0 
\end{align}  therefore 
\begin{align}
1_{n\times 1}^T \cdot x = 0
\end{align}
\end{remark}

\section{Vertex retrieval}
%
%
%
%

In order to efficiently retrieve the distinguished vertex $x^{\star}$ we analyze the tangent cone in both polytopes $\mathcal{P}_{C_S,r_S^2}$ and $\mathcal{P}_{C_S(\zeta), r(\zeta)^2} \cap \mathcal{P}_{C_S(-\zeta), r(-\zeta)^2}$ and prove that the tangent cone in the distinguished vertex of the intersection of the perturbed polytopes is strictly included in the tangent cone of the unperturbed polytope. 

Assume that the polytope $\mathcal{P}_{C_S,r_S^2}$  is simplicial and let $k_1, ..., k_n$ be the active facets at $x^{\star}$. Note that since $1_{n\times 1}^T \cdot x^{\star} = 0$, this vertex is not moved, hence it belongs to the perturbed polytopes as well. The tangent cones for the unperturbed polytope and intersection of perturbed polytopes are given by
\begin{align}
\mathcal{K} &= \left\{ x \in \mathbb{R}^n \biggr | \begin{bmatrix} 2 \cdot (C_S - C_{k_1})^T\\ \vdots \\ 2 \cdot (C_S - C_{k_n})^T\end{bmatrix} \cdot x \preceq 0_{n \times 1} \right\} \nonumber \\
\mathcal{K}_{\cap} &= \left\{ x   \biggr | \begin{bmatrix} 2 \cdot (C_S \pm \zeta \cdot \frac{1_{n\times 1}}{\sqrt{n}} - C_{k_1})^T\\ \vdots \\ 2 \cdot (C_S \pm \zeta \cdot \frac{1_{n\times 1}}{\sqrt{n}} - C_{k_n})^T\end{bmatrix} \cdot x \preceq 0_{n \times 1} \right\}
\end{align}
\begin{lemma} \label{L3.1}
Let $x $ be on an edge of $\mathcal{K}_{\cap}$, then $x$ belongs  to a subspace of dimension at most $\left\lfloor\frac{n}{2}\right\rfloor + 1$
\end{lemma}
\begin{proof}
Let $x \neq 0_{n \times 1}$ be a point on the edge of the the tangent cone of the intersection of the perturbed polytopes then, since it belongs to $n-1$ facets, assume w.l.o.g. 
\begin{align}\label{E3.1}
\begin{bmatrix} 2 \cdot (C_S \pm \zeta \cdot \frac{1_{n\times 1}}{\sqrt{n}} - C_{k_2})^T\\ \vdots \\ 2 \cdot (C_S \pm \zeta \cdot \frac{1_{n\times 1}}{\sqrt{n}} - C_{k_n})^T\end{bmatrix} \cdot x = 0_{n \times 1}
\end{align} then if $x$ would have to meet at least $n$ of the $2\cdot n$ facets. In the worst case, it is w.l.o.g. determiend by the matrix 
\begin{align}
\begin{bmatrix} 2 \cdot (C_S + \zeta \cdot \frac{1_{n\times 1}}{\sqrt{n}} - C_{k_1})^T\\ 2 \cdot (C_S - \zeta \cdot \frac{1_{n\times 1}}{\sqrt{n}} - C_{k_1})^T \\ \vdots \\ 2 \cdot (C_S + \zeta \cdot \frac{1_{n\times 1}}{\sqrt{n}} - C_{k_n})^T\\  2 \cdot (C_S - \zeta \cdot \frac{1_{n\times 1}}{\sqrt{n}} - C_{k_p})^T\end{bmatrix} \cdot x = 0_{n \times 1}
\end{align}  where $p \geq \left \lfloor \frac{n}{2} \right\rfloor$, therefore the rank of the linear space $x$ belongs to is at most $n - p \leq \left \lfloor \frac{n}{2} \right\rfloor + 1$.  
\end{proof}

\subsection{Vertex retrieval with  $\left \lfloor \frac{n}{2}\right\rfloor + 1$}
In order to retrieve the distinguished vertex we return to the initial hypothesis of having $\left \lfloor \frac{n}{2} \right\rfloor +1$ such hyperplanes.

If $\left \lfloor \frac{n}{2} \right\rfloor$ planes are given, let $\mathcal{P}$ denote the unperturbed polytope and $\mathcal{P}_{k,\pm}$ the perturbed polytope along the normals to the given planes. Note that $x^{\star}$ belongs to all these planes. Applying Theorem \ref{T2.9} we get $\mathcal{P}_{k,+} \cap \mathcal{P}_{k,-} \subseteq \mathcal{P}$ hence
\begin{align}
\hat{\mathcal{P}} = \bigcap_{k=1}^{q} \mathcal{P}_{k,\pm} \subseteq \mathcal{P}
\end{align} where $q = \left \lfloor \frac{n}{2} +1 \right\rfloor$. Letting $d_i$ be the normal for the $i$'th hyperplane, the tangent cone in $x^{\star}$ is given similarly to (\ref{E3.1}) by 
\begin{align}
\mathcal{K}_{i,\cap} &= \left\{ x   \biggr | \begin{bmatrix} 2 \cdot (C_S \pm \zeta \cdot d_i - C_{k_1})^T\\ \vdots \\ 2 \cdot (C_S \pm \zeta \cdot d_i - C_{k_n})^T\end{bmatrix} \cdot x \preceq 0_{n \times 1} \right\}
\end{align} For a point $x \in \hat{\mathcal{P}} \cap \partial \mathcal{P}$ it is necessary that it belongs to the frontier of all $\mathcal{P}_{k,\pm}$ simultaneously.  

Letting $x \in \bigcap_{i = 1}^q \partial \mathcal{K}_{i,\cap}$ follows that $x$ meets
\begin{align}
\begin{bmatrix} \begin{bmatrix} 2 \cdot (C_S + \zeta \cdot d_1 - C_{k_1})^T\\ 2 \cdot (C_S - \zeta \cdot d_1 - C_{k_1})^T \\ \vdots \\ 2 \cdot (C_S + \zeta \cdot d_1 - C_{k_n})^T\\  2 \cdot (C_S - \zeta \cdot d_1 - C_{k_p})^T\end{bmatrix}  \\ \vdots \\ \begin{bmatrix} 2 \cdot (C_S + \zeta \cdot d_q - C_{k_1})^T\\ 2 \cdot (C_S - \zeta \cdot d_q - C_{k_1})^T \\ \vdots \\ 2 \cdot (C_S + \zeta \cdot d_q - C_{k_n})^T\\  2 \cdot (C_S - \zeta \cdot d_q - C_{k_p})^T\end{bmatrix}  \end{bmatrix} \cdot x = 0_{p \cdot q \times n }
\end{align}
However the system matrix, under the stated assumptions has full rank, therefore $x = 0_{n \times 1}$ as Lemma \ref{L3.1} also suggests, since $x$ meets $q$ linear spaces of dimension at most $q$ with $q \geq \frac{n}{2}$.  

The main algorithmic result is given in the next theorem. 

\begin{theorem}\label{T3.2}
 Under Assumptions \ref{A1} and \ref{A2}, there exists an algorithm of polynomial arithmetic complexity in $m,n$ that finds $x^{\star}$, the solution to Problem \ref{P1}.
\end{theorem}

\begin{proof}
From Theorem \ref{T2.9} it follows that 
\begin{align}
\hat{\mathcal{P}} \subseteq \mathcal{P} 
\end{align} and the tangent cone of $ \hat{\mathcal{P}}$ in the distinguished point $x^{\star}$ is included in the tangent cone of $\mathcal{P}$. Since 
\begin{align}
x^{\star} \in \hat{\mathcal{P}} \cap \mathcal{P} 
\end{align} we propose the following retrieval procedure. 

Suppose 
\begin{align}
\mathcal{P}= \left\{ x | A_k^T\cdot x + B_k \leq 0, \hspace{0.3cm} k = 1, \hdots, m \right\}
\end{align}

 Then, for every $k = 1, \hdots, m$ solve the linear program
\begin{align}
\mathop{\text{argmax}}_{x \in  \hat{\mathcal{P}} } \hspace{0.5cm} A_k^T\cdot x.
\end{align} This requires exactly $m$ linear programs, which each can be solved in polynomial arithmetic complexity because $\hat{\mathcal{P}}$ is given as an intersection of at most $2 \cdot \left \lfloor \frac{n}{2}\right\rfloor + 2$ polytopes with $m$ facets. 

Since $x^{\star}$ is a vertex of $\mathcal{P} = \mathcal{P}_{C_S(0), r(0)^2}$, at least one defining inequality $A_k^T \cdot x + B_k \leq 0$ is active at $x^{\star}$. By Theorem \ref{T2.9} and Lemma \ref{L3.1}, $x^{\star}$ is the unique solution of 
\begin{align}
\max \hspace{0.3cm} \left\{ A_k^T \cdot x | x \in \hat{\mathcal{P}} \cap \mathcal{P}\right\} 
\end{align} when $A_k$ is the normal vector to a facet defining the tangent cone in $x^{\star}$. 

\end{proof}

\section{Numerical example}
For ilustrative purpose we give a $\mathbb{R}^2$ simulation of the proposed theory. In our example we consider $n =2, m= 5, \rho = 2, R = 3$ and obtain the set $\mathcal{Q}$ as in Figure \ref{F1}. The discs centers are $C_k = \rho \cdot \begin{bmatrix} \cos(2\cdot \pi \cdot t_k)\\ \sin(2\cdot \pi \dot t_k)\end{bmatrix}$ where $t_k \in \left\{\frac{1}{4}, \frac{2}{4}, \frac{2.8}{4}, \frac{6.5}{8} \right\}$ for all $k \in \overline{1,5}$. 

\begin{figure}[h]
\includegraphics[scale=0.5]{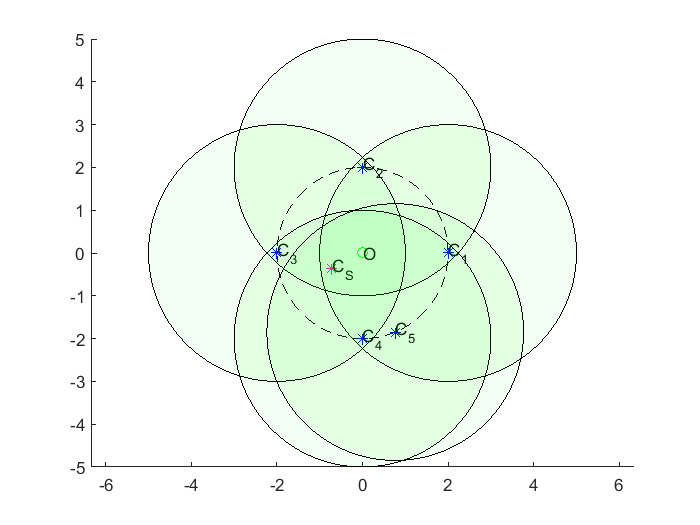}
\caption{The obtained ball polyhedra can be seen as the intersection of the five discs}
\label{F1}
\end{figure}

Next, we let $t_S = \frac{2.3}{4}$ and take $C_S = 0.4\cdot  \rho \cdot \begin{bmatrix} \cos(2\cdot \pi \cdot t_S)\\ \sin(2\cdot \pi \dot t_S)\end{bmatrix}$, $r = 1$, $\epsilon = 0.5$ and obtain the max-indicator polytopes $\mathcal{P}_{C_S(\zeta), r^2(\zeta)}$ for $\zeta \in \left\{\frac{-\epsilon}{\sqrt{2}}, 0, \frac{\epsilon}{\sqrt{2}} \right\}$. 

\begin{figure}[h]
\includegraphics[scale=0.5]{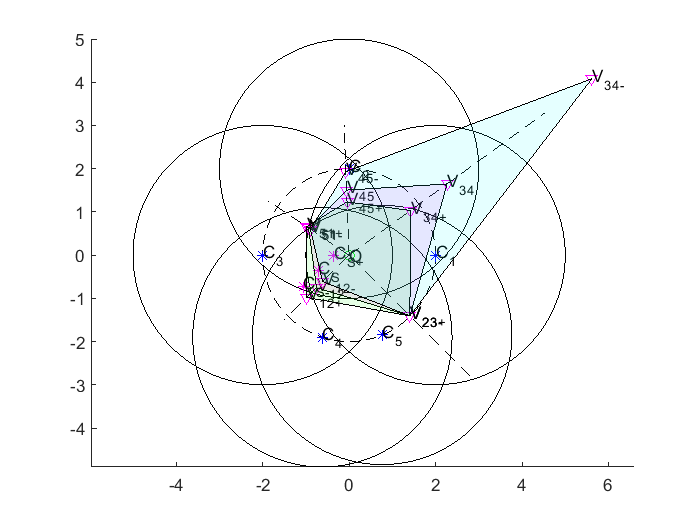}
\caption{The obtained max-indicator polytopes. One can see that the vertices slide on the their individual lines. Moreover, with magenta color one has $\mathcal{P}_{C_S(0), r(0)^2}$, with green  $\mathcal{P}_{C_S(\zeta), r(\zeta)^2}$ and with cyan $\mathcal{P}_{C_S(-\zeta), r(-\zeta)^2}$ for $\zeta = \frac{\epsilon}{\sqrt{2}}$. }
\label{F2}
\end{figure}

A close-up shown the presented theory, i.e. the unique vertex can be retrieved by LP.

\begin{figure}[h]
\includegraphics[scale=0.5]{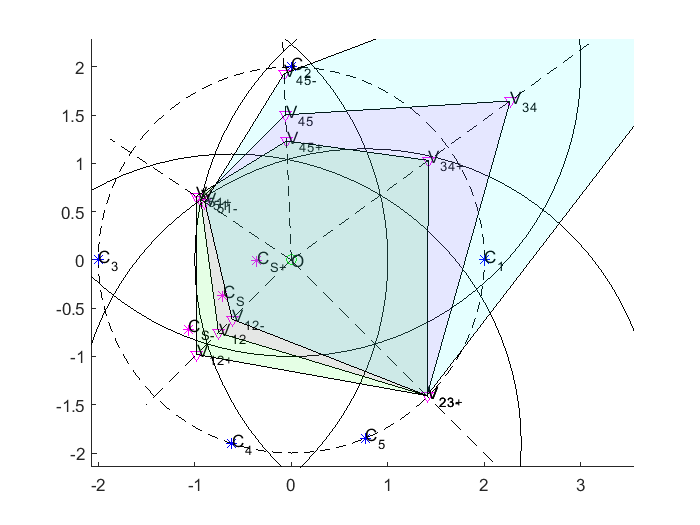}
\caption{The obtained max-indicator polytopes. The unique vertex on the hyperplane $V_{23-} = V_{23} = V_{23+}$ is not changed by the symmetric perturbation. Note the tangent cone inclusion.}
\label{F3}
\end{figure} 
\section{Conclusion}
%
%

We have presented a geometric algorithm that recovers, in polynomial arithmetic complexity the unique vertex of a ball polyhedron lying on a prescribed hyperplane
through the origin. The ball polyhedron is realized as the intersection of
equal-radius balls whose centers lie on a common sphere centered at the
origin, with the origin itself belonging to the convex hull of the centers.
Under the standing uniqueness and non-degeneracy assumptions, the algorithm
constructs a one-parameter family of auxiliary max-indicator polytopes,
exploits the radial motion of their vertices under a controlled translation
of the quadratic center, and isolates the distinguished vertex by a linear
number of linear programs.

\end{document}